\documentclass[12pt]{amsart}

\usepackage[utf8]{inputenc}
\usepackage[USenglish]{babel}
\usepackage{amsmath,amsthm,amssymb,amsfonts}
\usepackage{mathtools}
\usepackage{mathrsfs}
\usepackage{bbm}
\usepackage{indentfirst}
\usepackage{xifthen}
\usepackage{enumitem}
\usepackage{float}
\usepackage{graphicx}
\usepackage{placeins}
\usepackage{caption}

\usepackage{tikz}
\usetikzlibrary{patterns}
\usepackage{tikz-cd}

\usetikzlibrary{decorations.pathreplacing,decorations.markings}
\usetikzlibrary{babel}

\tikzset{                        
    symbol/.style={
        ,draw=none
        ,every to/.append style={%
            edge node={node [sloped, allow upside down, auto=false]{$#1$}}}
    }
}

\usepackage[breaklinks=true, bookmarksopenlevel=1, bookmarksdepth=2]{hyperref}

\allowdisplaybreaks

\newtheorem{thm}{Theorem}[section]

\newtheorem{lem}[thm]{Lemma}

\theoremstyle{definition}
\newtheorem{defn}[thm]{Definition}

\theoremstyle{remark}
\newtheorem{rem}[thm]{Remark}

\newtheorem*{ntt}{Notation}

\numberwithin{equation}{section}

\newcommand{\Z}{\mathbb{Z}}      
\newcommand{\R}{\mathbb{R}}      

\newcommand{\norm}[1]{\left\lVert #1\right\rVert} 

\newcommand\restr[2]{{           
  \left.\kern-\nulldelimiterspace #1%
  \right|_{#2}%
 }}

\newcommand\minus{               
  \setbox0=\hbox{-}%
  \vcenter{%
    \hrule width\wd0 height \the\fontdimen8\textfont3%
  }%
}

\newcommand{\ki}{\mathrm{K}} 
\newcommand{\kn}{\mathrm{N}} 
\newcommand{\ul}[1]{\mathbf{#1}}

\restylefloat{table}

\begin{document}

\title{Knights are $24/13$ times faster than the king}%
\author{Christian T\'afula}%
\address{D\'epartment de Math\'ematiques et Statistique, %
 Universit\'e de Montr\'eal, %
 CP 6128 succ Centre-Ville, %
 Montreal, QC H3C 3J7, Canada}%
\email{christian.tafula.santos@umontreal.ca}%

\subjclass[2020]{11B13, 11B75}%
\keywords{sumsets, chess knight}%

\begin{abstract}
 On an infinite chess board, how much faster can the knight reach a square when compared to the king, in average? More generally, for coprime $b>a \in \mathbb{Z}_{\geq 1}$ such that $a+b$ is odd, define the $(a,b)$-knight and the king as
 \[ \mathrm{N}_{a,b}= \{(a,b), (b,a), (-a,b), (-b,a), (-b,-a), (-a,-b), (a,-b), (b, -a)\}, \]
 \[ \mathrm{K}=\{(1,0), (1,1), (0,1), (-1,1), (-1,0), (-1,-1), (0,-1), (1,-1)\}  \subseteq \mathbb{Z}^2, \]
 respectively. One way to formulate this question is by asking for the average ratio, for $\mathbf{p}\in \mathbb{Z}^2$ in a box, between $\min\{h\in \Z_{\geq 1} ~|~ \mathbf{p}\in h\mathrm{N}\}$ and $\min\{h\in \Z_{\geq 1} ~|~ \mathbf{p}\in h\mathrm{K}\}$, where $hA = \{\ul{a}_1+\cdots+\ul{a}_h ~|~ \ul{a}_1,\ldots, \ul{a}_h \in A\}$ is the $h$-fold sumset of $A$. We show that this ratio equals $2(a+b)b^2/(a^2+3b^2)$.
\end{abstract}
\maketitle

\section{Introduction}
 Let $A\subseteq \Z^2$ be a finite set. For each $\ul{p}\in \Z^2$, we are interested in determining the smallest $h\geq 1$ for which we can write $\ul{p} = \ul{a}_1 + \cdots +\ul{a}_h$, where $\ul{a}_i \in A$ for $1\leq i\leq h$ are not necessarily distinct. Writing $hA = \{\ul{a}_1+\cdots+\ul{a}_h ~|~ \ul{a}_1,\ldots, \ul{a}_h \in A\}$ for the $h$-fold sumset of $A$, we define $A(0,0) := 0$ and, for $(x,y)\neq (0,0)$,
 \begin{equation}
  A(x,y) := \min\{h\geq 1 ~|~ (x,y) \in hA\}. \label{Axy}
 \end{equation}
 
 The study of the \emph{size} of $hA$ goes back to Khovanskii \cite{kho92}, who showed that $|hA|$ is given by a polynomial in terms of $h$ for $h$ sufficiently large (cf. Nathanson--Ruzsa \cite{natruz02} for a more combinatorial proof). In another direction, Granville--Shakan--Walker \cite{grasha20, grashawal23} studied the \emph{structure} of $hA$, showing that, roughly speaking, for every large $h$, every element that ``could be'' in $hA$ is, in fact, in $hA$.
 
 In this note, we will study the behaviour of $A(x,y)$ for a particular class of sumsets. Thinking of $\Z^2$ as an infinite chess board, a finite set $A$ may be thought of as a \emph{piece} placed at the origin, being able to move only to $\ul{a}\in A$. Then, in two moves, the piece is able to reach every point in $2A$, and so on. We say that $A$ is:
 
 \begin{itemize}   
  \item \emph{Primitive} if $A(x,y)$ is well-defined for every $x,y\in\Z$;\smallskip
   
  \item \emph{Symmetric} if $(a, b)\in A$ implies $(\delta_1 a, \delta_2 b)$, $(\delta_1 b, \delta_2 a) \in A$ for every choice of $\delta_1,\delta_2 \in \{-1, +1\}$;\smallskip
 \end{itemize}
 
  \begin{ntt}
  For real functions $f,g:\R_{>0} \to \R$, we write $f(x) = O(g(x))$ if there is $M>0$ such that $f(x) \leq Mg(x)$ for every large $x$.
 \end{ntt}
 
 \subsection{The king and the \texorpdfstring{$(a,b)$}{(a,b)}-knight}
  The two pieces that will concern us in this note are the following:
  
  \begin{enumerate}[label=\alph*)]
   \item The \emph{king} $\ki = \{(1,0), (1,1), (0,1), (-1,1), (-1,0), (-1,-1), (0,-1), (1,-1)\}$ is the smallest symmetric piece with $(1,0), (1,1)\in \ki$.
   
   \FloatBarrier
   \begin{figure}[!htb]
   \centering
   \noindent
   \begin{minipage}{0.39\textwidth}\begin{center}
    \bigskip
    \begin{tikzpicture}[scale=0.75, every node/.style={scale=0.75}]
     \draw[help lines, ystep=1, xstep=1, black!80] (-3,-3) grid (4,4);
     
     \foreach \x in {-3,-1,1,3}
     \foreach \y in {-2,0,2}{
      \fill[pattern=north east lines, pattern color=black!75] (\x,\y)--(\x,\y +1)--(\x +1,\y +1)--(\x +1,\y);
      \fill[pattern=north east lines, pattern color=black!75] (\y,\x)--(\y,\x +1)--(\y +1,\x +1)--(\y +1,\x);
     }
     
     \foreach \x in {-1,0,1}
     \foreach \y in {-1,0,1}{
      \ifthenelse{\x = 0}{
       \draw[->, black, thick] (0+.5,0+.5)--(\x +.5, 1.2*\y +.5) }{
       \ifthenelse{\y = 0}{
       \draw[->, black, thick] (0+.5,0+.5)--(1.2*\x +.5, \y +.5) }{
       \draw[->, black, thick] (0+.5,0+.5)--(\x +.5, \y +.5) }};
       
       \draw[black, thick] (\x,\y)--(\x,\y +1)--(\x +1, \y +1)--(\x +1, \y)--(\x,\y);
     }
     
     \draw[fill=white] (0 +.5,0 +.5) circle (1em);
     \draw[black] (0 +.5,0 +.5) node {\large{$\ki$}};
    \end{tikzpicture}
    \end{center}\end{minipage} \hspace{1em}
    \begin{minipage}{0.39\textwidth}\begin{center}
    \bigskip
    \begin{tikzpicture}[scale=0.75, every node/.style={scale=0.75}]
     \draw[help lines, ystep=1, xstep=1, black!80] (-3,-3) grid (4,4);

     \fill[pattern=north east lines, pattern color=black!20] (-3,-3)--(-3,3+1)--(3+1,3+1)--(3+1, -3);
     \fill[pattern=north east lines, pattern color=black!40] (-2,-2)--(-2,2+1)--(2+1,2+1)--(2+1, -2);
     \fill[pattern=north east lines, pattern color=black!80] (-1,-1)--(-1,1+1)--(1+1,1+1)--(1+1, -1);
     \fill[white] (-0,-0)--(-0,0+1)--(0+1,0+1)--(0+1, -0);
     
     \foreach \x in {3,2,1,0}{
      \draw[black, thick] (-\x,-\x)--(-\x,\x +1)--(\x +1, \x +1)--(\x +1, -\x)--(-\x,-\x);
     }
     
     \foreach \x in {0,...,3}
     \foreach \y in {0,...,3}
     \foreach \sg in {-1, 1}
     \foreach \sgg in {-1, 1}{
      \draw[white, fill=white] (\sg*\x +.5,\sgg*\y +.5) circle (.8em);
      \draw[black] (\sg*\x +.5,\sgg*\y +.5) node {\ifthenelse{\x > \y}{\Large{$\x$}}{\large{$\y$}}};
     }
    \end{tikzpicture}
   \end{center}\end{minipage}
   \begin{center}
    \begin{minipage}{0.75\textwidth}\begin{center}
     \captionsetup{singlelinecheck=off, width=\textwidth}
     \caption{The king's movements (left) and $\ki(x,y)$ (right).}
    \end{center}\end{minipage}
   \end{center}
   \label{fig1}
   \end{figure}

   \item For $a,b \in \Z_{\geq 1}$, we define the \emph{$(a,b)$-knight} $\kn_{a,b}$ by the set of moves   
   \begin{align*}
    \kn_{a,b} := \{\,& (b,a),\, (a,b),\, (-a,b),\, (-b,a),\\
    &(-b,-a),\, (-a,-b),\, (a,-b),\, (b,-a)\,\}; 
   \end{align*}
   in other words, $\kn_{a,b}$ is the smallest symmetric piece with $(a,b)\in \kn_{a,b}$. The usual chess knight is the $(1,2)$-knight, which we call just \emph{knight} and denote it by $\kn$.\phantom{\qedhere}
   
   \FloatBarrier
   \begin{figure}[!htb]
   \centering
   \noindent
   \begin{minipage}{0.39\textwidth}\begin{center}
    \bigskip
    \begin{tikzpicture}[scale=0.75, every node/.style={scale=0.75}]
     \draw[help lines, ystep=1, xstep=1, black!80] (-3,-3) grid (4,4);
     
     \foreach \x in {-3,-1,1,3}
     \foreach \y in {-2,0,2}{
      \fill[pattern=north east lines, pattern color=black!75] (\x,\y)--(\x,\y +1)--(\x +1,\y +1)--(\x +1,\y);
      \fill[pattern=north east lines, pattern color=black!75] (\y,\x)--(\y,\x +1)--(\y +1,\x +1)--(\y +1,\x);
     }
     
     \foreach \x in {-1,1}
     \foreach \y in {-1,1}{
      \draw[->, black, thick] (0+.5,0+.5)--(\x*2 +.5, \y*1 +.5);
      \draw[->, black, thick] (0+.5,0+.5)--(\x*1 +.5, \y*2 +.5);
       
      \draw[black, thick] (\x*2,\y*1)--(\x*2,\y*1 +1)--(\x*2 +1, \y*1 +1)--(\x*2 +1, \y*1)--(\x*2,\y*1);
      \draw[black, thick] (\x*1,\y*2)--(\x*1,\y*2 +1)--(\x*1 +1, \y*2 +1)--(\x*1 +1, \y*2)--(\x*1,\y*2);
     }
     
     \draw[fill=white] (0 +.5,0 +.5) circle (1em);
     \draw[black] (0 +.5,0 +.5) node {\large{$\kn$}};
    \end{tikzpicture}
    \end{center}\end{minipage} \hspace{1em}
    \begin{minipage}{0.39\textwidth}\begin{center}
    \bigskip
    \begin{tikzpicture}[scale=0.75, every node/.style={scale=0.75}]
     \draw[help lines, ystep=1, xstep=1, black!80] (-3,-3) grid (4,4);
     
     \draw[black] (0 +.5,0 +.5) node {\Large{$0$}};
     
     \foreach \x in {-1,1}
     \foreach \y in {-1,1}{
      \fill[pattern=north east lines, pattern color=black!80] (\x*2,\y*1)--(\x*2,\y*1 +1)--(\x*2 +1, \y*1 +1)--(\x*2 +1, \y*1)--(\x*2,\y*1);
      \fill[pattern=north east lines, pattern color=black!80] (\x*1,\y*2)--(\x*1,\y*2 +1)--(\x*1 +1, \y*2 +1)--(\x*1 +1, \y*2)--(\x*1,\y*2);
      
      \draw[white, fill=white] (\x*2 +.5,\y*1 +.5) circle (.8em);
      \draw[black] (\x*2 +.5,\y*1 +.5) node {\Large{$1$}};
      \draw[white, fill=white] (\x*1 +.5,\y*2 +.5) circle (.8em);
      \draw[black] (\x*1 +.5,\y*2 +.5) node {\Large{$1$}};
     }
     
     \foreach \x in {-3,-1,1,3}
     \foreach \y in {-3,-1,1,3}{
      \fill[pattern=north east lines, pattern color=black!40] (\x,\y)--(\x,\y +1)--(\x +1, \y +1)--(\x +1, \y)--(\x,\y);

      \draw[white, fill=white] (\x +.5,\y +.5) circle (.8em);
      \draw[black] (\x +.5,\y +.5) node {\Large{$2$}};
     }
     \foreach \x in {-2,0,2}
     \foreach \y in {-2,0,2}{
      \ifthenelse{\x = 0 \OR \y = 0}{\ifthenelse{\x = 0 \AND \y = 0}{}{
      \fill[pattern=north east lines, pattern color=black!40] (\x,\y)--(\x,\y +1)--(\x +1, \y +1)--(\x +1, \y)--(\x,\y);

      \draw[white, fill=white] (\x +.5,\y +.5) circle (.8em);
      \draw[black] (\x +.5,\y +.5) node {\Large{$2$}};}}{}
     }
     
    \foreach \sg in {-1,1}{
     \foreach \x in {-2,2}{
      \fill[pattern=north east lines, pattern color=black!20] (\x,\sg*3)--(\x,\sg*3 +1)--(\x +1, \sg*3 +1)--(\x +1, \sg*3)--(\x,\sg*3);

      \draw[white, fill=white] (\x +.5,\sg*3 +.5) circle (.8em);
      \draw[black] (\x +.5,\sg*3 +.5) node {\Large{$3$}};
     }
     \foreach \x in {-3,3}{
      \fill[pattern=north east lines, pattern color=black!20] (\x,\sg*2)--(\x,\sg*2 +1)--(\x +1, \sg*2 +1)--(\x +1, \sg*2)--(\x,\sg*2);

      \draw[white, fill=white] (\x +.5,\sg*2 +.5) circle (.8em);
      \draw[black] (\x +.5,\sg*2 +.5) node {\Large{$3$}};
     }
    }
     \foreach \x in {-3,-1,0,1,3}
     \foreach \y in {-3,-1,0,1,3}{
      \ifthenelse{\x = 0 \OR \y = 0}{\ifthenelse{\x = 0 \AND \y = 0}{}{
      \fill[pattern=north east lines, pattern color=black!20] (\x,\y)--(\x,\y +1)--(\x +1, \y +1)--(\x +1, \y)--(\x,\y);

      \draw[white, fill=white] (\x +.5,\y +.5) circle (.8em);
      \draw[black] (\x +.5,\y +.5) node {\Large{$3$}};}}{}
     }
     
     \foreach \x in {-2,2}
     \foreach \y in {-2,2}{
      \fill[pattern=north east lines, pattern color=black!0] (\x,\y)--(\x,\y +1)--(\x +1, \y +1)--(\x +1, \y)--(\x,\y);

      \draw[white, fill=white] (\x +.5,\y +.5) circle (.8em);
      \draw[black] (\x +.5,\y +.5) node {\Large{$\mathbf{4}$}};
     }

     \foreach \x in {3,2,1,0}{
      \draw[black, thick] (-\x,-\x)--(-\x,\x +1)--(\x +1, \x +1)--(\x +1, -\x)--(-\x,-\x);
     }
     
    \end{tikzpicture}
   \end{center}\end{minipage}
   \begin{center}
    \begin{minipage}{0.75\textwidth}\begin{center}
     \captionsetup{singlelinecheck=off, width=\textwidth}
     \caption{The knight's movements (left) and $\kn(x,y)$ (right).} 
    \end{center}\end{minipage}
   \end{center}
   \label{fig2}
   \end{figure}
   
   Not all $(a,b)$-knights are primitive. In fact, for $\kn_{a,b}$ to be primitive, it is necessary and sufficient that $\gcd(a,b) = 1$ and $a+b$ be odd. To see this, color $\Z^2$ like a chess board (i.e., paint $(x,y)$ white if $2\mid x+y$, and black otherwise). The necessary direction is then easy: $\gcd(a,b) \mid \gcd(x,y)$ for every point $(x,y)$ accessible to $\kn_{a,b}$, and if $a+b$ is even then $\kn_{a,b}$ never accesses black points. For the sufficient direction, note that since $\kn_{a,b}$ changes colors every move, it suffices to show that it can access all the white points, and by symmetry, it suffices to show that it accesses $(2,0)$. Since $(b,a) + (b,-a) = (2b,0)$ and $(a,b) + (a,-b) = (2a,0)$, the $(a,b)$-knight can access every point of the form $(2(ax+by), 0)$ for $x,y\in\Z$; which, since $\gcd(a,b) = 1$, implies that $\kn_{a,b}$ can access $(2,0)$.
  \end{enumerate}
  
  By the symmetries of $\kn_{a,b}$, to understand the behaviour of $\kn_{a,b}(x,y)$ it suffices to study $x\geq y\in\Z_{\geq 0}$, where $\ki(x,y) = x$. We will show the following:
  
  \begin{thm}\label{wklm}
  Let $b > a \geq 1$ be integers with $\gcd(a,b)=1$ and $a+b$ odd, and let $x \geq y \in \Z_{\geq 0}$.
  \begin{enumerate}[label=\textnormal{(\roman*)}]
   \item If $\displaystyle y \leq \frac{a}{b}\,x$, then $\displaystyle \kn_{a,b}(x,y) = \frac{x}{b} + O(b)$. \medskip
   
   \item If $\displaystyle y > \frac{a}{b}\,x$, then $\displaystyle \kn_{a,b}(x,y) = \frac{x+y}{a + b} + O(b)$.
  \end{enumerate}
 \end{thm}
 
 In Subsection \ref{dnst}, we describe the distribution of $\kn/\ki$.

 \subsection{Average velocity in a box}
 Each finite set $A$ induces a metric $d_{A}(\ul{p},\ul{q}) := A(\ul{q}-\ul{p})$ for $\ul{p},\ul{q}\in\Z^2$. The king's metric coincides with the one induced by the $\max$ norm
 \[ \norm{(x,y)}_{\infty} = \max\{|x|,|y|\}, \]
 and thus we equip $\Z^2$ with this metric. For $h\geq 1$, write
 \[ \mathcal{B}_h := \{\ul{p}\in\Z^2 ~|~ \norm{\ul{p}}_{\infty} \leq h\}, \qquad \mathcal{B}_h^{*} := \mathcal{B}_h\setminus\{(0,0)\} \]
 for the ball and punctured ball of radius $h$, respectively. Note that $\mathcal{B}_h = \bigcup_{\ell=1}^{h} \ell\ki$ and $\partial\mathcal{B}_h = \{\ul{p}\in\Z^2 ~|~ \norm{\ul{p}}_{\infty} = h\} = h\ki \setminus \bigcup_{\ell=0}^{\ell-1} \ell\ki$. We have $|\partial \mathcal{B}_h| = 8h$ and $|\mathcal{B}^{*}_h| = 4 h(h+1)$.
 
 What is the average value of $A(x,y)$ in $\mathcal{B}_h$? For instance, the king $\ki$ is such that $\ki(x,y) =\ell$ if and only if $(x,y) \in \partial \mathcal{B}_{\ell}$; hence,
 \[ \frac{1}{|\mathcal{B}^{*}_h|} \sum_{\ul{p}\in \mathcal{B}^{*}_h} \ki(\ul{p}) = \frac{1}{4h(h+1)} \sum_{\ell=1}^{h} \ell\cdot 8\ell = \frac{2h}{3} + \frac{1}{3}. \]
 Thus, we consider the following notion of velocity, which can be understood intuitively as how fast the king $\ki$ sees the piece $A$ moving (see Remark \ref{gens}):
 
 \begin{defn}[Velocity]\label{boxvel}
  For a finite primitive set $A\subseteq \Z^{2}$, the \emph{average velocity} $v = v_{\ki}$ of $A$ is given by
  \[ v(A) := \lim_{h\to +\infty} \frac{2h}{3}\Bigg(\frac{1}{|\mathcal{B}_h|} \sum_{\ul{p}\in \mathcal{B}_h} A(\ul{p})\Bigg)^{-1}. \]
 \end{defn}
 
 The number $v(A)$ may be thought of as controlling how fast $A$ spreads through $\mathcal{B}_h$. Intuitively, from Theorem \ref{wklm}, one might conclude that the knight is almost, although not quite, \emph{twice} as fast as the king. Points of the type $(x,0)$, for example, can be accessed by the knight in around $x/2$ moves, while points of the form $(x,x)$ can be accessed in around $2x/3$ moves. We will show that
 \[ \frac{\sum_{\ul{p}\in \mathcal{B}_h} \ki(\ul{p})}{\sum_{\ul{p}\in \mathcal{B}_h} \kn(\ul{p})} \ \xrightarrow{h\to +\infty} \ \frac{24}{13}; \]
 in other words, the ``not quite'' is quantified by $2/13$. More generally:
 \begin{thm}\label{knab}
  Let $b > a  \geq 1$ be integers with $\gcd(a,b)=1$ and $a+b$ odd. Then:
  \[ v(\kn_{a,b}) = \frac{2(a+b)b^2}{a^2+3b^2}. \]
 \end{thm}
 
 See Remark \ref{fbkn} for a consequence of Theorem \ref{knab} when one takes $a$, $b$ to be consecutive Fibonacci numbers --- which we call \emph{Fiboknights}.

\section{Knights in \texorpdfstring{$\Z^2$}{Z\^{}2}}
 We start with a lemma estimating how long the $(a,b)$-knight takes to access a point in $\mathcal{B}_{a+b}$.
 
 \begin{lem}\label{Bab}
  Let $b> a \geq 1$ be integers with $\gcd(a,b)=1$ and $a+b$ odd. For every $(x,y)\in \mathcal{B}_{a+b}$, we have $\kn_{a,b}(x,y) = O(b)$ uniformly for $a,b$.
 \end{lem}
 \begin{proof}
  Since $\gcd(a,b) = 1$, for every $1\leq k\leq b$ there are $x,y\in \Z$ with $ax+by=k$, and we can select $x,y$ such that $|x|\leq b$, $|y|\leq a$. Hence, since $\kn_{a,b}$ is symmetric,
  \[ (2k,0) = x\big((a,b) + (a,-b)\big) + y\big((b,a) + (b,-a)\big) \]
  is accessible in $2(|x|+|y|)\leq 2(a+b) < 4b$ moves, and so are the points $(-2k,0)$, $(0,2k)$, $(0,-2k)$. This implies that every point in $\mathcal{B}_{a+b}$ with even coordinates is accessible in $O(b)$ moves. By symmetry, it then suffices to show $\kn_{a,b}(1,0) = O(b)$.
  
  Suppose that $a$ is even (so $b$ is odd). Then, the point $(1-a,-b) \in \mathcal{B}_{a+b}$ has even coordinates, and so is accessible in $O(b)$ moves. Therefore, so is $(1,0) = (1-a,-b) + (a,b)$. The case when $a$ is odd (so $b$ is even) is similar.
 \end{proof}
 
\subsection{Proof of Theorem \ref{wklm}}
 We prove the parts separately.\medskip
 
 \noindent
 $\bullet\text{ \underline{Part} (i):}$
 Let $\ell := \lfloor x/b\rfloor$, so that $\ell b \leq x < (\ell+1)b$ and $0\leq y < (\ell+1)a$. Because $x\geq \ell b$, we have $\kn_{a,b}(x,y) \geq \ell$. On the other hand, for each integer $0\leq k \leq \ell/2$,
 \[ \big(\ell - k\big)(b,a) + k(b,-a) = (\ell b, (\ell-2k)a), \]
 so all the points in $\mathcal{S}_{(\ell b, \ell a)} := \{(\ell b, (\ell-2k) a) ~|~ 0\leq k \leq \ell/2 \}$ are accessible in $\ell$ moves or less. All the points $(x,y)$ with $\ell b \leq x < (\ell+1)b$ and $0\leq y < (\ell+1)a$ are at distance\footnote{With respect to the $\max$ norm.} at most $a+b$ from $\mathcal{S}_{(\ell b, \ell a)}$. Since, by Lemma \ref{Bab}, $\kn_{a,b}$ accesses all the points of $\mathcal{B}_{a+b}$ in $O(b)$ moves, it follows that $\kn_{a,b}(x,y) \leq \ell + O(b)$.\medskip
 
 \noindent
 $\bullet\text{ \underline{Part} (ii):}$
 Let $t,u\in\R_{\geq 0}$ be such that $(x,y) = t(a,b) + u(b,a)$, so that $\kn_{a,b}(x,y) \geq t+u$. Since
 \[ \begin{pmatrix} a & b \\ b & a \end{pmatrix} \begin{pmatrix} t \\ u \end{pmatrix} = \begin{pmatrix} x \\ y \end{pmatrix} \iff \frac{1}{b^2-a^2}\begin{pmatrix} -a & b \\ b & -a \end{pmatrix} \begin{pmatrix} x \\ y \end{pmatrix} = \begin{pmatrix} t \\ u \end{pmatrix}, \]
 we have $t = (by-ax)/(b^2-a^2)$, $u = (bx-ay)/(b^2-a^2)$ (both strictly positive, because $y/x > a/b$), and hence
 \begin{align*}
  \kn_{a,b}(x,y) &\geq \frac{(b-a)(x+y)}{b^2-a^2} = \frac{x+y}{a+b}.
 \end{align*}
 On the other hand, $\lfloor t\rfloor (a,b) + \lfloor u\rfloor (b,a) = (x,y) + \ul{r}$, where $\ul{r} \in \mathcal{B}_{a+b}$. Since, by Lemma \ref{Bab}, $\kn_{a,b}$ accesses all the points of $\mathcal{B}_{a+b}$ in $O(b)$ moves, it follows that $\kn_{a,b}(x,y) \leq \lfloor t\rfloor + \lfloor u\rfloor + O(b) = \frac{x+y}{a+b} +O(b)$. \hfill$\square$
 
 \subsection{Distribution of \texorpdfstring{$\kn/\ki$}{K/N}}\label{dnst}
  It follows from Theorem \ref{wklm} that, for $x\geq y \in \Z_{\geq 0}$, the ratio $\frac{\kn_{a,b}(x,y)}{\ki(x,y)}$ lies essentially in between $\frac{1}{b}$ and $\frac{2}{a+b}$:
  \[ \frac{\kn_{a,b}(x,y)}{\ki(x,y)} =
  \begin{cases}
   \displaystyle \frac{1}{b} + O\bigg(\frac{b}{x}\bigg) &\text{if } \dfrac{y}{x}\leq \dfrac{a}{b},\\[1em]
   \displaystyle \frac{1}{a+b}\bigg(1+\frac{y}{x}\bigg) + O\bigg(\frac{b}{x}\bigg) &\text{if } \dfrac{y}{x}> \dfrac{a}{b}.
  \end{cases} \] 
  Analysing this ratio in the box $\mathcal{B}_h$, one can study the \emph{distribution} of $\kn_{a,b}/\ki$ via the real function
  \[ D_{a,b}(t) := \lim_{h\to+\infty} \frac{\#\{(x,y) \in \mathcal{B}_h ~|~ \frac{\kn_{a,b}(x,y)}{\ki(x,y)} \leq t \}}{|\mathcal{B}_{h}|}. \]
  Both $\kn_{a,b}$ and $\ki$ are symmetric. Therefore, since $\frac{1}{a+b}(1+\frac{y}{x}) \leq t$ if and only if $\frac{y}{x} \leq (a+b)t-1$, and the proportion of points in $\mathcal{B}_h \cap \{(x,y)\in \Z_{\geq 0} ~|~ x\geq y\}$ with $\frac{y}{x} \leq u$ equals $\frac{2}{h(h+1)}\sum_{x=1}^{h} \sum_{y=1}^{\lfloor ux\rfloor} 1 = u + O(1/h)$, we have
  \begin{equation} D_{a,b}(t) = \begin{cases}
                   0 &\text{if } t< \frac{1}{b}, \\
                   (a+b)t-1 &\text{if } \frac{1}{b} \leq t \leq \frac{2}{a+b}, \\
                   1 &\text{if } t > \frac{2}{a+b}.
                  \end{cases} \label{distrib} \end{equation}
                  
\subsection{Proof of Theorem \ref{knab}}
 By the symmetries of $\kn_{a,b}(x,y)$, we have
 \begin{equation}
  \lim_{h\to +\infty} \frac{3}{2h}\Bigg(\frac{1}{|\mathcal{B}_h|} \sum_{\ul{p}\in \mathcal{B}_h} \kn_{a,b}(\ul{p})\Bigg) = \lim_{h\to +\infty} \frac{3}{2h}\Bigg(\frac{2}{h(h+1)} \sum_{\substack{x,y \in \Z_{\geq 0} \\ 1\, \leq\, y\, \leq\, x\, \leq h}} \kn_{a,b}(x,y) \Bigg), \hspace{-1em} \label{limitao}
 \end{equation}
 so it suffices to prove the existence and calculate the right-hand side.
 
 By Theorem \ref{wklm}, we have
 \begin{align*}
  \sum_{\substack{x,y \in \Z_{\geq 0} \\ 1\, \leq\, y\, \leq\, x\, \leq h}} \kn_{a,b}(x,y) &= \sum_{x=1}^{h} \bigg\lfloor \frac{a}{b}\, x\bigg\rfloor \frac{x}{b} + \sum_{x=1}^{h}\sum_{\substack{y=1 \\ y/x\, >\, a/b}}^{x} \frac{x+y}{a+b} + O(bh^2) \\
  &= \sum_{x=1}^{h}\Bigg(\frac{a}{b^2} + \frac{1}{a + b} \sum_{\substack{y=1 \\ y/x\, >\, a/b}}^{x} \bigg(1 + \frac{y}{x}\bigg)\frac{1}{x}\Bigg) x^2 + O(bh^2).
 \end{align*}
 
 Since
 \begin{align*}
  \sum_{\substack{y=1 \\ y/x\, >\, a/b}}^{x} \bigg(1 + \frac{y}{x}\bigg)\frac{1}{x} &= \frac{1}{x}\Bigg(\sum_{\substack{y=1 \\ y/x\, >\, a/b}}^{x} 1\Bigg) + \frac{1}{x^2}\Bigg(\sum_{\substack{y=1 \\ y/x\, >\, a/b}}^{x} y\Bigg) \\
  &= \bigg(1-\frac{a}{b}\bigg) + \frac{1}{2}\bigg(1-\frac{a^2}{b^2}\bigg) + O\bigg(\frac{1}{x}\bigg),
 \end{align*}
 it follows that:
 \begin{align*}
  \sum_{\substack{x,y \in \Z_{\geq 0} \\ 1\, \leq\, y\, \leq\, x\, \leq h}} \kn_{a,b}(x,y) = \Bigg(\frac{a}{b^2} + \frac{1}{a + b} \bigg(\bigg(1-\frac{a}{b}\bigg) + \frac{1}{2}\bigg(1-\frac{a^2}{b^2}\bigg)\bigg) \Bigg) \frac{h(h+1)(2h+1)}{6}& \\
  +\, O(bh^2).&
 \end{align*}
 Plugging this into the limit $v(\kn_{a,b})$, we obtain
 \begin{align*}
  v(\kn_{a,b}) &= \lim_{h\to +\infty} \frac{2h}{3}\Bigg(\frac{2}{h(h+1)} \sum_{\substack{x,y \in \Z_{\geq 0} \\ 1\, \leq\, y\, \leq\, x\, \leq h}} \kn_{a,b}(x,y) \Bigg)^{-1} \\
  &= \Bigg(\frac{a}{b^2} + \frac{1}{a + b} \bigg(\bigg(1-\frac{a}{b}\bigg) + \frac{1}{2}\bigg(1 - \frac{a^2}{b^2}\bigg) \bigg) \Bigg)^{-1} \\
  &= \frac{2}{3} \Bigg(\frac{2a^2 + 2ab}{3b^2} + \bigg(1-\frac{a}{b}\bigg)\bigg(1 + \frac{a}{3b}\bigg) \Bigg)^{-1}(a+b) \\
  &= \frac{2}{3}\bigg(1 + \frac{1}{3}\, \frac{a^2}{b^2} \bigg)^{-1}(a+b) = \frac{2(a+b)b^2}{a^2+3b^2},
 \end{align*}
 concluding the proof.\hfill$\square$

\section{Remarks}
 \begin{rem}
  One checks that calculating the average using \eqref{distrib} agrees, in fact, with (the inverse of) Theorem \ref{knab}:
  \begin{align*}
   \mathbb{E}\bigg(\frac{\kn_{a,b}}{\ki}\bigg) &:= \int_{0}^{+\infty} (1-D_{a,b}(t))\,\mathrm{d}t = \frac{1}{b} + \int_{1/b}^{2/(a+b)} (2-(a+b)t)\,\mathrm{d}t \\
   &= \frac{1}{b} + \frac{2(b-a)}{(b+a)b} - \frac{(b-a)(a+3b)}{2(a+b)b^2} = \frac{a^2+3b^2}{2(a+b)b^2}.
  \end{align*}
 \end{rem}

 \begin{rem}[On generality]\label{gens}
  The choice of the box $\mathcal{B}_h$ in Definition \ref{boxvel} is not generic, and different expanding regimes will give different answers for the ratio. In general, let $d\geq 2$ and $A\subseteq \Z^{2}$ be primitive set, and suppose that the origin $\ul{0}$ lies inside the convex hull $\mathcal{H}(A)$ of $A$. Write $A_{\ul{0}} = A\cup\{\ul{0}\}$. By Khovanskii's theorem \cite[Corollary 1]{kho92}, we have $|hA_{\ul{0}}| = \mathrm{vol}(\mathcal{H}(A))\,h^d + O(h^{d-1})$ and
  \[ |hA_{\ul{0}}\setminus (h-1)A_{\ul{0}}| = d\,\mathrm{vol}(\mathcal{H}(A))\, h^{d-1} + O(h^{d-2}), \]
  where $\mathrm{vol}(\mathcal{H}(A))$ denotes the $d$-volume of the convex hull of $A$. Thus,
  \begin{align*}
   \frac{1}{|hA_{\ul{0}}|} \sum_{\ul{p}\in hA_{\ul{0}}} A(\ul{p}) = \frac{1}{|hA_{\ul{0}}|} \sum_{\ell=1}^{h} \sum_{\ul{p}\in \ell A_{\ul{0}}\setminus (\ell-1)A_{\ul{0}}} A(\ul{p}) = \frac{dh}{d+1} + O(1).
  \end{align*}
  Given a finite primitive set $B\subseteq \Z^{d}$, we define the velocity of $B$ \emph{relative to $A$} as
  \[ v_A(B) := \lim_{h\to+\infty} \bigg(1+\frac{1}{d}\bigg) h\, \Bigg(\frac{1}{|hA_{\ul{0}}|} \sum_{\ul{p}\in hA_{\ul{0}}} B(\ul{p}) \Bigg)^{-1}. \]
  It would be interesting to calculate the velocity of generalized knights with respect to the generalized king $\mathrm{K}^{d} = \{\ul{p}\in\Z^{d} ~|~ \norm{p}_{\infty} = 1\}$, or velocities with respect to other pieces such as the \emph{taxicab} $\mathrm{T} := \{\ul{p} = (x,y)\in \Z^{2} ~|~ \norm{\ul{p}}_{1} := |x|+|y| = 1\} = \{(1,0), (0,1), (-1,0), (0,-1)\}$.
 \end{rem}
 

 \begin{rem}[Fiboknights]\label{fbkn}
  Fibonacci numbers $F_0 =1$, $F_1 = 1$, $F_n = F_{n-1} + F_{n-2}$ (for $n\geq 2$) satisfy the property that $F_{3n}$ is even, $F_{3n+1}$, $F_{3n+2}$ are odd, and $\gcd(F_n, F_{n+1})=1$. Define the \emph{$n$-th Fiboknight} as 
  \[ \mathrm{FN}_n = \kn_{F_{n+1},F_{n+2}},\]
  so that the usual knight is the first Fiboknight. By the properties of Fibonacci numbers, $\mathrm{FN}_n$ is only primitive for $n$ such that $3\nmid n$.
  
  Let $k\geq 1$, and let $n\to \infty$ through $n\in\Z_{\geq 1}$ for which $\mathrm{FN}_n$, $\mathrm{FN}_{n+k}$ are primitive. Then, by Theorem \ref{knab}, writing $\phi = \frac{1+\sqrt{5}}{2}$ for the golden ratio, we have
  \begin{align*}
   \lim_{\substack{n\to \infty \\ 3\nmid n,\, n+k}} \frac{v(\mathrm{FN}_{n+k})}{v(\mathrm{FN}_n)} &= \lim_{\substack{n\to \infty \\ 3\nmid n,\, n+k}} \frac{\displaystyle \frac{2(F_{n+k+1} + F_{n+k+2})F_{n+k+2}^2}{F_{n+k+1}^2 + 3F_{n+k+2}^2}}{\displaystyle \frac{2(F_{n+1} + F_{n+2})F_{n+2}^2}{F_{n+1}^2 + 3F_{n+2}^2}} \\
   &= \lim_{\substack{n\to \infty \\ 3\nmid n,\, n+k}} \frac{F_{n+k+3}}{F_{n+3}}\frac{F_{n+k+2}^2}{F_{n+2}^2} \frac{(F_{n+1}^2 + 3F_{n+2}^2)}{(F_{n+k+1}^2 + 3F_{n+k+2}^2)} \\
   &= \phi^k\, \phi^{2k}\, \frac{1+3\phi^2}{\phi^{2k} + 3\phi^{2k+2}} \\
   &= \phi^{k}.
  \end{align*}
  In particular, the ratio of the velocity of consecutive Fiboknights (which can only be of the form $\mathrm{FN}_{3n+1}$, $\mathrm{FN}_{3n+2}$) converges to $\phi$. In general, for fixed $m,k\geq 1$,
  \begin{align*}
   \lim_{\substack{n\to \infty \\ \text{primitive}}} \frac{v(\mathrm{N}_{F_{n+k}, F_{n+m+k}})}{v(\mathrm{N}_{F_{n}, F_{n+m}})} &= \frac{\displaystyle\ \frac{2(\phi^{k} + \phi^{m+k}) \phi^{2(m+k)}}{\phi^{2k} + 3\phi^{2(m+k)}}\ }{\displaystyle \frac{2(1 + \phi^{m}) \phi^{2m}}{1 + 3\phi^{2m}}} = \phi^{k}.
  \end{align*}
 \end{rem}
  
\addtocontents{toc}{\protect\setcounter{tocdepth}{0}}
\section*{Acknowledgements}
 I thank the anonymous referees for the insightful comments.
 
\addtocontents{toc}{\protect\setcounter{tocdepth}{1}}


\end{document}